\documentclass[preprint,12pt]{elsarticle}
\usepackage{amsthm,amsmath,amssymb}
\usepackage[colorlinks=true,citecolor=black,linkcolor=black,urlcolor=blue]{hyperref}
\usepackage{graphicx}
\usepackage{float}
\usepackage{mathrsfs}
\usepackage{mathtools}
\usepackage{algorithm}
\usepackage{algpseudocode}
\usepackage{comment}
\usepackage{epstopdf}
\usepackage{epsfig}
\usepackage{subfigure}

\newtheorem*{claim*}{Claim}
\usepackage{tikz}
\usetikzlibrary{positioning}
\usepackage{float}
\usetikzlibrary{shapes.geometric} 
\usepackage{caption}

\theoremstyle{plain}
\newtheorem{theorem}{Theorem}
\newtheorem{lemma}[theorem]{Lemma}
\newtheorem{corollary}[theorem]{Corollary}
\newtheorem{proposition}[theorem]{Proposition}

\newtheorem{observation}[theorem]{Observation}

\theoremstyle{definition}
\newtheorem{definition}[theorem]{Definition}

\newtheorem{problem}[theorem]{Problem}

\theoremstyle{remark}

\title{The polynomial characterization of tope graphs of the lopsided sets}

\author{Xuan Zheng, Yan-Ting Xie,
Shou-Jun Xu$^*$\\
\small School of Mathematics and Statistics, Gansu Center for Applied Mathematics, Lanzhou University, Lanzhou, Gansu 730000, China\\}

\begin{document}
\begin{abstract}
The cube polynomial $C_G(x)$ generates the number of $k$-cubes on a graph $G$. As a subclass of partial cubes, the tope graphs of lopsided sets (LOPs) generalize daisy cubes and median graphs. 
In this paper, we prove that every tope graph of a LOP shares its cube polynomial with some daisy cube, thereby answering affirmatively a problem posed earlier by the authors. 
Furthermore, we present explicit expressions for the cube polynomials of tope graphs of LOPs: $C_G(x)=f_\mathcal{K}(x+1)$, where $f_\mathcal{K}(x)$ is the $f$-polynomial of the cubical complex $\mathcal{K}$ of $G$.
Finally, we show that a class $\mathcal{G}$ of partial cubes is the class of tope graphs of LOPs if and only if the following equivalent conditions hold: (a) $\mathcal{G}$ is the maximal pc-minor-closed class such that every graph
in $\mathcal{G}$ has the same cube polynomial as some daisy cube; (b) for every
$G\in\mathcal{G}$, each antipodal subgraph of $G$ has the same cube polynomial as some
daisy cube.

\end{abstract}
\maketitle
\noindent
$^*$Corresponding author, e-mail: \texttt{shjxu@lzu.edu.cn}

\section{Introduction}
A graph polynomial is a graph invariant with values in a polynomial ring, usually a subring of real polynomial ring $\mathbb{R}[x]$. Graph polynomials provide important structural information of graphs and
have been extensively studied as graph invariants. However, since non-isomorphic graphs may share identical polynomial invariants, a graph polynomial generally does not uniquely characterize a graph, such as the characteristic polynomial \cite{S73}, the chromatic polynomial \cite{DK05}, and the Tutte polynomial \cite{EC08}, Therefore,
an important research problem is to investigate polynomially equivalent
graphs, i.e., the class of non-isomorphic graphs that have the same certain graph polynomials. More precisely, for two graphs $G,G'$ and a graph polynomial $P$, if $P_G(x)=P_{G'}(x)$, we say that $G$ and $G'$ are {\em $P$-equivalent}.

In the present paper, we mainly focus on cube polynomials, which are defined as follows:  
\begin{definition}[\cite{BSKR1}]\label{def01}
For a graph $G$, let $c_k(G)$ ($k\geq 0$) be the number of induced subgraphs of $G$ isomorphic to $Q_k$ in $G$. The \textit{cube polynomial} of $G$, $C_G(x)$, is defined as
   \[C_G(x):=\sum_{k\ge0}c_k(G)x^k.\]
\end{definition}

For two graphs $G$ and $G'$, we say that they are {\em $C$-equivalent} if their cube polynomials coincide, that is, if $C_G(x)=C_{G'}(x)$. The cube polynomial was introduced by Bre\v{s}ar, Klav\v{z}ar, and \v{S}krekovski \cite{BSKR1} and has since attracted considerable attention; see, for example, \cite{BSKR2,xfx24}. Most existing studies on cube polynomials concern partial cubes, with particular emphasis on median graphs.

The \textit{$n$-dimensional hypercube} (or \textit{$n$-cube}) $Q_n$ is the graph with vertex set
$B^n \,:=\, \big\{ u_1u_2\cdots u_n \mid u_i\in\{0,1\},\ 1\le i\le n \big\}.$
Two vertices are adjacent if and only if the corresponding binary strings differ in precisely one position.
A {\em partial cube} is a graph isomorphic to an isometric subgraph of a hypercube $Q_n$. A median graph is a special partial cube that for every its three vertices, there exists a unique vertex that lies on the shortest paths between each pair of the three vertices simultaneously. Let $X$ be a subset of $V(Q_n)$ and $0^n$ be the all-zero vertex of $Q_n$. The \textit{daisy cube} $Q_n(X)$ generated by $X$ is the induced subgraph of $Q_n$ spanned by all vertices lying on some shortest path from $0^n$ to a vertex in $X$. The simplex graph $S(G)$ of a graph $G$ is defined as the graph whose vertices are the cliques of $G$ (including the empty set), with two vertices being adjacent if, as cliques of $G$, they differ in exactly one vertex. Xie and Xu \cite{xx25} proved that the class of simplex graphs is exactly the intersection of the class of median graphs and the class of daisy cubes. 
About the cube polynomials of median graphs and simplex graphs, the authors proved the following result.

\begin{lemma}[\cite{zxx26}]\label{S=M}
For any median graph $G$, there exists a simplex graph $G'$ such that $G$ is $C$-equivalent to $G'$.
\end{lemma}

In this paper, we generalize the median graphs in Lemma \ref{S=M} to the context of the tope graphs of lopsided sets. The lopsided sets were introduced by Lawrence \cite{l83} in 1983. In 1995, the concept of lopsided sets was rediscovered in the context of extremal combinatorics and named \textit{ample sets} \cite{bcdk06}. 
 Lopsided sets provided an important class of the sample compression conjecture \cite{AD97}, which is one of most important open problem in computational learning theory.
The tope graphs of lopsided sets are known to be partial cubes and the important superclass of median graphs and daisy cubes \cite{m18}. We denote the class of tope graphs of lopsided sets by $\mathcal{G}_{\mathrm{LOP}}$. 
Recently, Knauer and Marc \cite{KKTM} established purely graph-theoretical characterizations of lopsided tope graphs via forbidden partial cube minors and antipodal subgraphs.
We consider the following problem proposed in our previous work \cite{zxx26}:

\begin{problem}[\cite{zxx26}]\label{pro02}
For any tope graph $G$ of a lopsided set, does there exist a daisy cube $G'$ satisfying
$$C_G(x)=C_{G'}(x)?$$
\end{problem}

We provide the following affirmative answer.
\begin{theorem}\label{L=D}
    For any tope graph $G$ of a lopsided set, there exists a daisy cube $G'$ satisfying
$$C_G(x)=C_{G'}(x).$$
\end{theorem}

From the proof of Theorem~\ref{L=D}, we derive an explicit formula for the cube polynomial of the tope graph of a lopsided set (see Section~3). To establish a polynomial characterization of the tope graphs of lopsided sets, recall the definition of the $f$-polynomial.

\begin{definition}[\cite{mw26}]\label{def01}

Let $\mathcal{K}$ be a $d$-dimensional simplicial complex. For $-1\leq i\leq d-1$, let
$f_i(\mathcal{K}):=\#\{\sigma\in\mathcal{K}\mid \dim(\sigma)=i\},$
where $f_{-1}(\mathcal{K})=1$. The vector
\[
(f_{-1}(\mathcal{K}),f_0(\mathcal{K}),\dots,f_{d-1}(\mathcal{K}))
\]
is called the \textit{$f$-vector} of $\mathcal{K}$. The associated \textit{$f$-polynomial} of $\mathcal{K}$ is defined by
\[
f_{\mathcal{K}}(x):=\sum_{i=0}^{d}f_{i-1}(\mathcal{K})x^i .
\]

\end{definition}

Daisy cubes can be viewed as graph-theoretic representations of simplicial complexes. Omitting the root, the distance polynomial of a daisy cube is an $f$-polynomial of the corresponding simplicial complex. Recently, we gave the polynomial characterization of daisy cubes as $C_G(x)=W_{G,0^n}(x+1)$ in \cite{zxx26}. Therefore, for any graph $G$, there exists a daisy cube $G'$ satisfying 
$C_G(x)=C_{G'}(x)$ if and only if  $C_G(x-1)$ of graph $G$ can be the $f$-polynomial of some simplicial complex. 
We obtain the following corollary:

\begin{corollary}\label{C=f-poly}
    If $G$ is a tope graph of a lopsided set $L$, $\mathcal{K}$ is its cubical complex, then 
$$C_G(x)=f_\mathcal{K}(x+1).$$
\end{corollary}

The above result states that when graph $G$ is the tope graph of a lopsided set, $C_G(x-1)$ is the $f$-polynomial of the cubical complex. This naturally raises the question of when $C_G(x-1)$ of graph $G$ is the $f$-polynomial of some simplicial complex. Is the tope graph of a lopsided set the largest partial cube class satisfying the above result? There is a negative answer. The counterexample is as follows. 

Let $G$ (see Figure~\ref{fig:Q4_double_minus}) be the graph obtained by taking the join of $Q_4^{--}$ and $C_6$, denoted by $Q_4^{--} \vee C_6$. As illustrated in Figure~\ref{fig:daisy cube}, there exists a daisy cube $D$ with the same cube polynomial as $G$, that is,
$$C_D(x)=C_G(x)=19+30x+12x^2.$$ However, 
$G$ is not a tope graph of a lopsided set.
Nevertheless, the tope graphs of lopsided sets remain the largest pc-minor closed class of partial cubes satisfying the property stated in Corollary~\ref {C=f-poly}, that is,

\begin{figure}[!t]
    \centering

\begin{tikzpicture}[
    x=0.5cm,
    y=0.5cm,
    vertex/.style={
        circle,
        draw=black,
        fill=black,
        line width=1.5pt,
        inner sep=0pt,
        minimum size=1.6mm
    },
    edge/.style={
        black,
        line width=0.9pt,
        line cap=round,
        line join=round
    }
]


\coordinate (v0001) at (2.25,7.5);

\coordinate (v0011) at (0,6);
\coordinate (v0101) at (3,6);
\coordinate (v1001) at (4.5,6);

\coordinate (v0111) at (0.75,4.5);
\coordinate (v1011) at (2.25,4.5);
\coordinate (v1101) at (5.25,4.5);

\coordinate (v0010) at (0,3);
\coordinate (v0100) at (3,3);
\coordinate (v1000) at (4.5,3);

\coordinate (v0110) at (0.75,1.5);
\coordinate (v1010) at (2.25,1.5);
\coordinate (v1100) at (5.25,1.5);

\coordinate (v1110) at (3,0);

\coordinate (A) at (6.5, 6);
\coordinate (B) at (8.25, 6);
\coordinate (C) at (9.5, 4.5);
\coordinate (D) at (8.25, 3);
\coordinate (E) at (6.5, 3);

\draw[edge] (v1101) -- (A) -- (B)-- (C) -- (D)-- (E) -- (v1101);

\node[vertex] at (A) {};

\node[vertex] at (B) {};
\node[vertex] at (C) {};
\node[vertex] at (D) {};

\node[vertex] at (E) {};

\draw[edge] (v0001) -- (v0011);
\draw[edge] (v0001) -- (v0101);
\draw[edge] (v0001) -- (v1001);

\draw[edge] (v0011) -- (v0010);
\draw[edge] (v0011) -- (v0111);
\draw[edge] (v0011) -- (v1011);

\draw[edge] (v0101) -- (v0100);
\draw[edge] (v0101) -- (v0111);
\draw[edge] (v0101) -- (v1101);

\draw[edge] (v1001) -- (v1000);
\draw[edge] (v1001) -- (v1011);
\draw[edge] (v1001) -- (v1101);

\draw[edge] (v0111) -- (v0110);
\draw[edge] (v1011) -- (v1010);
\draw[edge] (v1101) -- (v1100);

\draw[edge] (v0010) -- (v0110);
\draw[edge] (v0010) -- (v1010);

\draw[edge] (v0100) -- (v0110);
\draw[edge] (v0100) -- (v1100);

\draw[edge] (v1000) -- (v1010);
\draw[edge] (v1000) -- (v1100);

\draw[edge] (v0110) -- (v1110);
\draw[edge] (v1010) -- (v1110);
\draw[edge] (v1100) -- (v1110);


\node[vertex] at (v0001) {};

\node[vertex] at (v0011) {};
\node[vertex] at (v0101) {};
\node[vertex] at (v1001) {};

\node[vertex] at (v0111) {};
\node[vertex] at (v1011) {};
\node[vertex] at (v1101) {};

\node[vertex] at (v0010) {};
\node[vertex] at (v0100) {};
\node[vertex] at (v1000) {};

\node[vertex] at (v0110) {};
\node[vertex] at (v1010) {};
\node[vertex] at (v1100) {};

\node[vertex] at (v1110) {};

\end{tikzpicture}
 \caption{The graph $Q_4^{--}\vee C_6$: the joint graph of the 4-cube with two antipodal vertices removed and a 6-cycle.}
    \label{fig:Q4_double_minus}
    
\hspace{0.5cm}

\begin{tikzpicture}
[ x=0.7cm,
    y=0.7cm,
dot/.style={circle, fill=black, draw=black, inner sep=0pt, minimum size=5pt},line/.style={black, thick}
]

\node[dot] (f) at (-0.9, 0) {};
\node[dot] (e) at (-1.8, 0) {};
\node[dot] (d) at (-2.7, 0) {};
\node[dot] (c) at (-3.6, 0) {};

\node[dot] (b) at (-4.5, 0) {};
\node[dot] (a) at (-5.4, 0) {};
\node[dot] (g) at (0, 0) {};
\node[dot] (h) at (0.9, 0) {};

\node[dot] (i) at (1.8, 0) {};
\node[dot] (j) at (2.7, 0) {};
\node[dot] (k) at (3.6, 0) {};
\node[dot] (l) at (4.5, 0) {};

\node[dot] (A) at (-3.9, -2) {};
\node[dot] (B) at (-2.6, -2) {};
\node[dot] (C) at (-1.3, -2) {};
\node[dot] (D) at (0, -2) {};
\node[dot] (E) at (1.3, -2) {};
\node[dot] (F) at (2.6, -2) {};
\node[dot] (G) at (-0.6, -4) {};
\draw[line] (A) -- (G);
\draw[line] (B) -- (G);
\draw[line] (C) -- (G);
\draw[line] (D) -- (G);
\draw[line] (E) -- (G);
\draw[line] (F) -- (G);

\draw[line] (a) -- (A);
\draw[line] (a) -- (C);

\draw[line] (b) -- (A);
\draw[line] (b) -- (D);

\draw[line] (c) -- (A);
\draw[line] (c) -- (E);

\draw[line] (d) -- (A);
\draw[line] (d) -- (F);

\draw[line] (e) -- (B);
\draw[line] (e) -- (C);

\draw[line] (f) -- (B);
\draw[line] (f) -- (D);

\draw[line] (g) -- (B);
\draw[line] (g) -- (E);

\draw[line] (h) -- (B);
\draw[line] (h) -- (F);

\draw[line] (i) -- (C);
\draw[line] (i) -- (E);

\draw[line] (j) -- (C);
\draw[line] (j) -- (F);

\draw[line] (k) -- (D);
\draw[line] (k) -- (E);

\draw[line] (l) -- (D);
\draw[line] (l) -- (F);

\end{tikzpicture}
\caption{The daisy cube $D$.}
    \label{fig:daisy cube}

\end{figure}

\begin{theorem}\label{pc-m-closed}
Let $\mathcal{G}$ be a pc-minor closed partial cube class. If for any  $G \in \mathcal{G}$, there is a daisy cube $D$ such that $C_G(x)=C_D(x)$, then $\mathcal{G}\subseteq\mathcal{G}_{\mathrm{LOP}}$.
\end{theorem}

Moreover, we establish a series of subgraph-local polynomial characterizations of tope graphs of lopsided sets.

\begin{theorem}\label{C_H=-1}
Let $G$ be a partial cube. Then the following assertions are equivalent:
\begin{enumerate}[(i)]
    \item $G$ is the tope graph of a lopsided set;
    \item For any antipodal subgraph $H$ of $G$, there exists a daisy cube $H_1$ such that $C_{H}(x)=C_{H_1}(x)$;
    \item For any antipodal subgraph $H$ of $G$, it satisfies $C_{H}(-1)=1$.

\end{enumerate}
\end{theorem}

The paper is organized as follows. In Section~2, we recall the relevant definitions and introduce the notation used throughout the paper. In Section~3, we collect some useful known lemmas and establish several auxiliary lemmas, and then use these results to prove our main theorems.

\section{Preliminaries}

Throughout this paper, unless stated otherwise, all graphs considered are undirected, finite, and simple. Let $G$ be a graph with vertex set $V(G)$ and edge set $E(G)$. If all pairs of vertices of a subgraph $H$ of $G$ that are adjacent in $G$ are also adjacent in $H$,
then $H$ is an \textit{induced subgraph}. 
For $u,v\in V(G)$, the {\em distance} between $u$ and $v$ is defined as the length of a shortest $u,v$-path, denoted by $d_G(u,v)$. We will omit the subscript $G$ if this causes no confusion. The greatest distance between any two vertices in $G$ is the {\em diameter} of $G$, denoted by $\mathrm{diam}(G)$. The \textit{interval} $I_G(u,v)$ between two vertices $u$ and $v$ is the set of vertices on the shortest paths between $u$ and $v$. 
For a subgraph $H$, if $d_H(u,v) =d_G(u,v)$ for all $u,v \in V(H)$, we say $H$ is an \textit{isometric subgraph}.  
A subgraph $H$ of $G$ is \textit{convex} if for all pairs of vertices in $ H$, all their shortest paths in $G$ stay in $ H$. For any $S\subseteq V(G)$, the smallest convex subgraph of $G$ containing $S$ is called the {\em convex hull} of $S$ and denoted by $\mathrm{conv}(S)$. For $u,v\in V(G)$, $\mathrm{conv}(\{u,v\})$ is denoted by $\mathrm{conv}(u,v)$ for shortly.  
A graph $G$ is called {\em partial cube} if it is isomorphic to an isometric subgraph of $Q_n$ for some $n$. Graham and Pollak \cite{GR71} introduced partial cubes in the study of interconnection networks.

The {\em Djokovi\'c-Winkler relation} (see \cite{dj73,w84}) $\Theta$ is a binary relation on $E(G)$ defined as follows: Let $e=uv$ and $f=xy$ be two edges in $G$, $e\,\Theta\,f\iff d(u,x)+d(v,y)\neq d(u,y)+d(v,x)$. Winkler \cite{w84} proved that a graph $G$ is a partial cube if and only if $G$ is bipartite and $\Theta$ is an equivalence relation on $E(G)$. 
Let $G$ be a partial cube. We call the equivalence class on $E(G)$ {\em $\Theta$-class}. For $e=uv\in E(G)$, we denote the $\Theta$-class containing $uv$ as $F_{uv}$, i.e., $F_{uv}:=\{f\in E(G) \mid f\,\Theta\,e\}$.  
Moreover, we denote $W_{uv}:=\{w\in V(G) \mid d_{G}(u,w)<d_{G}(v,w)\}$. This set is a convex halfspace of $G$. The {\em isometric dimension} of a partial cube $G$ is the smallest integer $n$ such that $G$ can be isometrically embedded into $Q_n$, denoted by $\mathrm{idim}(G)$. This dimension coincides with the number of $\Theta$-classes of $G$  \cite{dj73}.

Let $H$ be a subgraph of $G$. If for any $v\in V(H)$, there exists a vertex $-_{H}v$ such that $\mathrm{conv}(v,-_Hv)=H$, we say that $-_Hv$ is the {\em antipode} of $v$ with respect to $H$. For a partial cube $G$ and $u,v\in V(G)$, $\mathrm{conv}(u,v)=I(u,v)$ since intervals in a partial cube are convex. It is easy to see that if a vertex has an antipode with respect to a subgraph $H$, it is unique. We call a subgraph $H$ of a partial cube $G$ \textit{antipodal} if every vertex $v$ of $H$ has an antipode with respect to $H$. By definition, we can see that antipodal subgraphs are convex. And furthermore, we obtain: 
\begin{observation}\label{obs:d=id}
    Let $H$ be an antipodal subgraph of a partial cube $G$. For any $v\in V(H)$,
    $$d_H(v,-_Hv)=\mathrm{idim}(H).$$
\end{observation}

We follow the terminology about pc-minor closed of
Chepoi, Knauer, and Marc~\cite{ckm20},
Knauer and Marc~\cite{KKTM}, and
Chepoi, Knauer, and Philibert~\cite{ckp20}.
For a $\Theta$-class $F_{uv}$ of $G$, an {\em elementary restriction}
consists of taking one of the complementary halfspaces $W_{uv}$ and
$W_{vu}$. More generally, a {\em restriction} is a subgraph of $G$
induced by the intersection of a set of (non-complementary) halfspaces
of $G$. A {\em contraction} is obtained from $G$ by contracting the edges of $F_{uv}$.

A partial cube $H$ is a {\em partial cube minor}, or
{\em pc-minor} for short, of $G$. 
if $H$ can be obtained from $G$ by a finite,
possibly empty, sequence of elementary restrictions and
contractions of $\Theta$-classes. 
A class $\mathcal{C}$ of partial cubes is said to be {\em closed under
taking pc-minors}, or simply {\em pc-minor closed}, if  $G\in\mathcal{C}$ and $H$ being a pc-minor of $G$ implies that $H\in\mathcal{C}$.

The following proposition shows that for a partial cube, its convex subgraphs are its pc-minors.
\begin{proposition}[\cite{ak16,b89,c86}]\label{pro:ConvexSubgraph}
Let $G$ be a partial cube and $H$ a subgraph of $G$. $H$ is a convex subgraph of $G$ if and only if $H$ is a restriction of $G$.
\end{proposition}

Let $0^n$ be the all-zero vertex of $Q_n$, and let $X\subseteq V(Q_n)$ be a nonempty subset. The \textit{daisy cube} $Q_n(X)$ generated by $X$, introduced by Klav\v zar and Mollard~\cite{SKM},
is the induced subgraph of $Q_n$ spanned by all vertices lying on some shortest path from $0^n$ to a vertex in $X$.
Equivalently, let $\preceq$ be the partial order on $B^n$ such that
$u_1u_2\cdots u_n \preceq v_1v_2\cdots v_n$ holds whenever $u_i\le v_i$ for all $i\in [n]$.
Then
\[
V(Q_n(X)) \,:=\, \bigl\{ u\in B^n \mid u\preceq x \text{ for some } x\in X \bigr\}.
\]
About the cube polynomials of daisy cubes, Klav\v zar and Mollard proved:
\begin{lemma}[\cite{SKM}]\label{C_G=(-1)}
If partial cube $G$ is a daisy cube, then $C_G(-1)=1.$
\end{lemma}
In combinatorics, a set system $\Delta$ is called an {\em abstract simplicial complex}
(or an {\em independence system}) if it is closed under taking subsets;
that is, if $X\in\Delta$ and $Y\subseteq X$, then $Y\in\Delta$.
In particular, daisy cubes can be viewed as graph-theoretic representations of abstract simplicial complexes.

Let $E$ be a finite set. We denote by
\[
\operatorname{Sign}(E):=\{-1,1\}^{E}
\]
the set of all maps from \(E\) to \(\{-1,1\}\), and by \(\mathbb{R}^{E}\) the set of all maps from \(E\) to \(\mathbb{R}\).
Let $\mathcal{L}$ be a subset of $\mathrm{Sign}(E)$ and $Y$ a subset of $E$. For $t\in\mathrm{Sign}(E-Y)$ and $s\in\mathrm{Sign}(E)$, if $s(e)=t(e)$ for all $e\in E-Y$, we say that $s$ is an {\em extension} of $t$. Denote 
$$\mathcal{L}^Y:=\{t\in \mathrm{Sign}(E-Y) \mid \mbox{every extension of }t\mbox{ belongs to }\mathcal{L}\},$$
and further, the abstract simplicial complex
$$\underline{\mathcal{X}}(\mathcal{L}):=\{Y\subseteq E \mid \mathcal{L}^Y\neq\emptyset\}$$ 
is called the {\em cubical complex} of $\mathcal{L}$.
$\mathcal{L}$ is called a {\em lopsided set} if for every $A\subseteq E$, either $A\in\underline{\mathcal{X}}(\mathcal{L})$, or $E-A\in\underline{\mathcal{X}}(\mathrm{Sign}(E)-\mathcal{L})$ \cite{bcdk06,l83}.
The {\em tope graph} of a lopsided set $\mathcal{L}$ on a finite set $E$, denoted by $G(\mathcal{L})$, is the graph whose vertex set is $\mathcal{L}$. Two vertices $f,g\in\mathcal{L}$ are adjacent if there exists exactly one $e\in E$ such that $f(e)\neq g(e)$ \cite{KKTM}.

For a graph $G$ and $u\in V(G)$, let $w_d(G)$ ($d \ge 0$) be the number of vertices of $G$ at distance $d$ from $u$. The \textit{distance polynomial} of $G$ with respect to $u$, $W_{G,u}(x)$, is defined as
    \[W_{G,u}(x):=\sum_{d\ge0}w_d(G)x^d.\]
    
Let $E$ be a set. the {\em power set} of $E$ is denoted by $2^E$, i.e., $2^E=\{S|S\subseteq E\}$. For $u\in B^n$, the \textit{weight} of $u$ is the number of 1s in word $u$, denoted as $w(u)$.

\section{Proofs of the main results}

Before proving Theorem~\ref{L=D}, we establish some auxiliary lemmas.
\begin{lemma}[\cite{l83}]\label{inher+com}
Let $E$ be a finite set, and $\mathcal{L}$ be a lopsided subset of $\mathrm{Sign}(E)$. Then 
\begin{enumerate}[(i)]
    \item $\mathrm{Sign}(E)-\mathcal{L}$ is lopsided;
    \item If $S$ is a subset of $E$ and $\pi_S \colon \mathbb{R}^E \to \mathbb{R}^S$ is the restriction of $f \in \mathbb{R}^E$ to $S$, then the image $\pi(\mathcal{L})$ of $\mathcal{L}$ is a lopsided subset of $\{-1,1\}^S$.
\end{enumerate}
\end{lemma}

\begin{lemma}\label{inher}
If $\mathcal{L}$ is lopsided, then $\mathcal{L}^Y$ is lopsided for any $Y \subseteq E$.
\end{lemma}

\begin{proof}
If $t \notin \mathcal{L}^Y$, then there exists an extension $s \in \{-1,1\}^E$ of $t$ that does not belong to $\mathcal{L}$. Thus, $s \in \{-1,1\}^E-\mathcal{L}$, $\pi_{E-Y}(s)=t$. Therefore, we have
    $$\mathcal{L}^Y=\{-1,1\}^{E-Y}-\pi_{E-Y}(\{-1,1\}^E-\mathcal{L}).$$
By Lemma~\ref{inher+com}, $\mathcal{L}^Y$ is lopsided.
\end{proof}

\begin{lemma}[\cite{SKM}]\label{daisy+distance}
If $G$ is a daisy cube, then $$C_G(x)=W_{G,0^n}(x+1)=\sum\limits_{v\in V(G)}(x+1)^{w(v)}.$$
\end{lemma}

With these lemmas in hand, we now prove Theorem~\ref{L=D}.

\begin{proof}[\textbf{Proof of Theorem \ref{L=D}}]
Let $G$ be a tope graph of a lopsided set $\mathcal{L}$.
We construct a bijection $g \colon 2^E \to \{0,1\}^E$ such that for every $S \in 2^E$, we have $g(S) = (g_1(S), g_2(S),\dots,g_{|E|}(S))$. Here, 
\[
g_i(S) =
\begin{cases}
1, & i \in S; \\
0, & i \notin S.
\end{cases}
\]
Obviously, for any $S \in 2^E$,
\begin{equation}\label{eq:weightg=size}
    w(g(S))=|S|.
\end{equation}

Let $X=g(\underline{\mathcal{X}}(\mathcal{L}))=\{g(S) \mid S\in\underline{\mathcal{X}}(\mathcal{L})\}$. 
Then $G':=Q_{|E|}(X)$ is a daisy cube generated by $X$ and $V(G)=X$, since $\underline{\mathcal{X}}(\mathcal{L})$ is a simplicial complex. By Lemma~\ref{daisy+distance}, we have

\[
C_{G'}(x)=W_{G',0^n}(x+1)=\sum_{g(S) \in X}(x+1)^{w(g(S))}.
\]
We now turn to the top graph $G$ of $\mathcal{L}$. Since $\mathcal{L}$ is lopsided, $\mathcal{L}^Y$ is also lopsided for any $Y \subseteq E$ by Lemma~\ref{inher}. By the definition of a lopsided set, we have
$|\mathcal{L}^Y|=|\underline{\mathcal{X}}(\mathcal{L}^Y)|$. Using the property of lopsided sets \cite{bcdk06} that $(\mathcal{L}^Y)^Z=\mathcal{L}^{Y \cup Z}$ for any $Y,Z\subseteq E$, we compute the size of the set $\mathcal{L}^Y$ as follows.
\begin{align*}
|\underline{\mathcal{X}}(\mathcal{L}^Y)|&=|\{Z \subseteq X-Y \mid (\mathcal{L}^Y)^Z\neq \varnothing \}|\\
&=|\{Z \subseteq X-Y \mid \mathcal{L}^{Y \cup Z} \neq \varnothing\}|\\
&=|\{S \subseteq \underline{\mathcal{X}}(\mathcal{L}) \mid Y \subseteq S\}|.
\end{align*}

Since the number of $|Y|$-cubes in $G$ associated with $Y$ is equal to the size of the set $\mathcal{L}^Y$, the cube polynomial of $G$ is as follows. 
\begin{align*}
C_G(x)&=\sum_{Y \subseteq E}|\mathcal{L}^Y|x^{|Y|}\\
    &=\sum_{Y \subseteq E}|\underline{\mathcal{X}}(\mathcal{L}^Y)|x^{|Y|}\\
    &=\sum_{S \in \underline{\mathcal{X}}(\mathcal{L})}\sum_{Y \subseteq S}x^{|Y|}\\
    &=\sum_{S \in \underline{\mathcal{X}}(\mathcal{L})}(x+1)^{|S|}.
\end{align*}
By Eq.\eqref{eq:weightg=size}, we have $C_G(x)=C_{G'}(x)$. This completes the proof.
\end{proof}

From the proof of Theorem~\ref{L=D}, we derive an explicit formula for the cube polynomial of the tope graph of a lopsided set. 
    
\begin{corollary}\label{explicit}
    If $G$ is a tope graph of a lopsided set $L$, $\mathcal{K}$ is its cubical complex, then 
$$C_G(x)=\sum_{S \in \mathcal{K}}(x+1)^{|S|}.$$
\end{corollary}

Let $H(E)$ denote the hypercube with vertex set
$\operatorname{Sign}(E)=\{-1,+1\}^{E}$, in which two vertices
are adjacent if and only if they differ in exactly one coordinate.
A face of $H(E)$ is a subcube obtained by fixing the signs on
a subset of $E$. We identify each face with its vertex set.
For a set $\mathcal{L}\subseteq \mathrm{Sign}(E)$, let
\begin{equation}\label{eq:fi}
    f'_i(\mathcal{L}):=\#\bigcup\{\mathcal{L}^Y \mid Y \subseteq E, \#Y=i\}.
\end{equation}

\begin{lemma}[\cite{bcdk06}]\label{face}
A nonempty set $\mathcal{L}\subseteq \operatorname{Sign}(E)$ is lopsided if and only if
\[
\sum_{i\ge 0} (-1)^i f'_i(\mathcal{L}\cap \mathcal{F}) = 1 
\]
holds for all faces $\mathcal{F}$ of $H(E)$ such that
$\mathcal{L}\cap\mathcal{F}\neq\varnothing$.
\end{lemma}
By definition of restriction, for a face $\mathcal{F}$ of $H(E)$, the tope graph $G(\mathcal{L}\cap \mathcal{F})$ of $\mathcal{L}\cap \mathcal{F}$ can be obtained by performing a sequence of corresponding elementary restrictions on $G(\mathcal{L})$. Combined with Proposition \ref{pro:ConvexSubgraph}, $G(\mathcal{L}\cap \mathcal{F})$ is a convex subgraph of $G(\mathcal{L})$. By the definition \eqref{eq:fi}, $f'_i(\mathcal{L}\cap \mathcal{F})$ is equal to the number of $i$-cubes of $G(\mathcal{L}\cap \mathcal{F})$, i.e., $c_i(G(\mathcal{L}\cap \mathcal{F}))$. Thus, Lemma \ref{face} can be expressed in graph-theoretical terms as follows.
\begin{lemma}\label{face'}
    A partial cube $G$ is a tope graph of a lopsided set if and only if 
    $$
    C_H(-1)=1
    $$
    holds for all convex subgraphs $H$ of $G$.
\end{lemma}
\begin{proof}[\textbf{Proof of Theorem \ref{pc-m-closed}}] 
For any $G \in \mathcal{G}$, since $\mathcal{G}$ is pc-minor closed, any convex subgraph $H$ of $G$ is also in $\mathcal{G}$. Then there is a daisy cube $D$ such that $C_H=C_D$. Thus, $C_H(-1)=C_D(-1)=1$. Therefore, $G$ is a tope graph of a lopsided set by Lemma~\ref{face'}, that is, $\mathcal{G}\subseteq \mathcal{G_{\mathrm{LOP}}}.$
\end{proof}

Before proving Theorem~\ref {C_H=-1}, we state a key lemma.
\begin{lemma}[\cite{KKTM}]\label{LOP-equ}
For a graph $G$, $G$ is the tope graph of a lopsided set if and only if $G$ is a partial cube and all its antipodal subgraphs are hypercubes.

\end{lemma}

\begin{proof}[\textbf{Proof of Theorem \ref{C_H=-1}}]
We establish the equivalence by verifying the necessary implications below.

(i)$\implies$(ii): Since $G$ is a tope graph of a lopsided set, by Lemma~\ref{LOP-equ}, for any antipodal subgraph $H$, $H$ is a hypercube. Since hypercubes are daisy cubes, (ii) holds upon setting $H_1=H$.

\smallskip

(ii)$\implies$(iii): For any antipodal subgraph $H$ of $G$, there exists a daisy cube $H_1$ with
$$
C_H(x)=C_{H_1}(x).
$$
It follows from Lemma~\ref{C_G=(-1)} that
$$
C_H(-1)=C_{H_1}(-1)=1.
$$

\smallskip

(iii)$\implies$(i): 
 By Lemma~\ref{LOP-equ}, we only need to prove that $H$ is a hypercube. Assume that $H$ is not a hypercube with $\mathrm{idim}(H)=l$. Let $Q_k$ be a $k$-cube which is an induced subgraph of $H$. Since $H$ is not a hypercube, $k<l$. Denote the vertex set of $Q_k$ by $V(Q_k)$. Since $H$ is antipodal, the antipodal set of $V(Q_k)$ is also in $H$, denoted by $\overline{V(Q_k)}$, i.e., $\overline{V(Q_k)}=\{-_Hv|v\in V(Q_k)\}$. Now, we show that $V(Q_k)\neq\overline{V(Q_k)}$. Since $\mathrm{idim}(H)=l$, Observation \ref{obs:d=id} gives $d_H(v,-_Hv)=l$. On the other hand, $\mathrm{diam}(Q_k)=k$. Hence, $-_Hv\not\in V(Q_k)$.  
 That is, $V(Q_k)\neq\overline{V(Q_k)}$. Since the two subgraphs induced by $V(Q_k)$ and $\overline{V(Q_k)}$ are isomorphic (both are $ k$-cubes), the same-dimensional cubes come in pairs. Thus, $$C_{H}(-1)=\sum_{k\geq0}(-1)^kc_k$$ is even, which yields a contradiction with $C_{H}(-1)=1.$
\smallskip

\end{proof}

\section{Conclusion and some remarks}
In this paper, we have characterized the class of pc-minor closed partial cubes for which there exists a daisy cube having the same cube polynomial, namely the class of tope graphs of lopsided sets. We also construct an example $Q_4^{--} \vee C_6$ for which there exists a daisy cube with the same cube polynomial, but it is not a tope graph of a lopsided set. This naturally raises the following problem.
\begin{problem}
   Characterize all partial cubes whose cube polynomials coincide with those of some daisy cube.
\end{problem}

\vskip 0.2 cm
\noindent{\bf Acknowledgements} This work was partially supported by the National Natural Science Foundation of China (No. 12671410).

\vskip 0.2 cm
\noindent{\bf Data Availability:} No datasets were generated or analysed during the current study.

\section*{Declarations}
\noindent{\bf Competing interests} The authors have no competing interests to declare that are relevant to the content of this article.

\end{document}